\documentclass[12pt]{amsart}
\usepackage{amssymb}
\usepackage[utf8]{inputenc}
\usepackage{enumerate}
\usepackage[english]{babel}
\def\RR{\mathbb R}
\newcommand{\set}[1]{\left\lbrace #1\right\rbrace}
\providecommand{\abs}[1]{\left\lvert#1\right\rvert}
\providecommand{\norm}[1]{\left\lVert#1\right\rVert}
\newcommand{\qtq}[1]{\quad\text{#1}\quad}

\newtheorem{theorem}{Theorem}[section]
\newtheorem{proposition}[theorem]{Proposition}
\newtheorem{lemma}[theorem]{Lemma}

\theoremstyle{remark}
\newtheorem*{remark}{Remark}
\newtheorem*{wpa}{Well posedness assumption}

\numberwithin{equation}{section}

\begin{document}
\title[Energy decay for evolution equations with delay]{Energy decay for evolution equations with delay feedbacks}

\author[V. Komornik]{Vilmos Komornik}
\address{College of Mathematics and Computational Science, Shenzhen Uni- versity, Shenzhen 518060, People’s Republic of China, and Département de mathématique\\
         Université de Strasbourg\\
         7 rue René Descartes\\
         67084 Strasbourg Cedex, France}
\email{komornik@math.unistra.fr}
\author[C. Pignotti]{Cristina Pignotti}
\address{Dipartimento di Ingegneria e Scienze dell'Informazione e Matematica, Universit\`a di L'Aquila, Via Vetoio, Loc. Coppito, 67010 L'Aquila, Italy}
\email{pignotti@univaq.it}
\thanks{The first author was supported by the grant NSFC No. 11871348. The research of the second author was partially supported by
GNAMPA 2018 project ``Analisi e controllo di modelli differenziali non lineari"(INdAM).
This work has been initiated during the first author's visit of the Department DISIM of Universit\`a di L'Aquila in December 2016.
He thanks the members of the department for their hospitality.}
\date{Version of 2019-02-20}

\begin{abstract}
We study abstract linear and nonlinear evolutionary systems with single or multiple delay feedbacks, illustrated by several concrete examples.
In particular, we assume that the operator associated with the undelayed part of the system generates an exponentially stable semigroup and that the delay damping coefficients are locally integrable in time.
A step by step procedure combined with Gronwall's inequality allows us to prove the existence and uniqueness of solutions.
Furthermore, under appropriate conditions we obtain exponential decay estimates.
\end{abstract}
\maketitle

\section{Introduction}\label{s1}

First we consider the evolution equation
\begin{equation}\label{11}
\begin{cases}
U'(t)=AU(t)+k(t)BU(t-\tau)\qtq{in}(0,\infty),\\
U(0)=U_0,\\
BU(t-\tau)=f(t)\qtq{for}t\in(0,\tau),
\end{cases}
\end{equation}
where $A$ generates an exponentially stable semigroup $(S(t))_{t\ge 0}$ in a Hilbert space $H$, $B$ is a  continuous linear operator of $H$ into itself, $k:[0,\infty)\to\RR$ is a function belonging to $L^1_{loc}([0,\infty);\RR)$, and $\tau>0$ is a delay parameter.
The initial data $U_0$ and $f$ are taken in $H$ and $C([0,\tau];H)$, respectively.
By the assumptions on the operator $A$ there exist two numbers $M,\omega>0$ such that
\begin{equation}\label{12}
\norm{S(t)}\le Me^{-\omega t}\qtq{for all}t\ge 0.
\end{equation}
Time delay effects often appear in applications to physical models, and it is well-known (\cite{BP, DLP}) that they can induce instability phenomena.
We are interested in showing that under some mild assumptions on $\tau, k$, the operator $B$ and the constants $M, \omega$ the system \eqref{11} is still exponentially stable.

Stability results for the above abstract  model have been recently obtained in \cite{JEE15, NicaisePignotti18} but only  for a constant delay feedback coefficient $k$.
In these papers  some nonlinear extensions are also  considered.
The arguments of \cite{JEE15, NicaisePignotti18} could be easily extended to $k(\cdot)\in L^\infty$ with  sufficiently small $\norm{k}_\infty$.
Recently, motivated by some applications, $L^1_{loc}$ damping coefficients $k$ have been considered, for instance of intermittent type (see \cite{Haraux, Pignotti, FP}), and stability estimates have been obtained for some particular models.
Here, our aim is to give a well-posedness result  and an exponential decay estimate for the general model \eqref{11} with a damping coefficient $k$ belonging only to $L^1_{loc}$.

As a non-trivial generalization, next we consider the case of multiple time-varying delays.
Namely, let $\tau_i:[0, +\infty)\rightarrow (0, +\infty),$ $i=1, \dots, l,$ be the time delays functions belonging to $W^{1,\infty}(0,+\infty)$.
We assume for each  $i=1, \ldots, l,$ that
\begin{equation}\label{13qqq}
0\le \tau_i(t)\le \overline\tau_i
\end{equation}
and
\begin{equation}\label{14qqq}
\tau_i^\prime (t)\le c_i<1.
\end{equation}
for a.e. $t>0$, with suitable constants $\overline\tau_i$ and $c_i$.
It follows  from \eqref{14qqq} that
\begin{equation*}
(t-\tau_i(t))^\prime = 1-\tau_i^\prime (t) >0, \quad \mbox{a.e.}\ t>0,
\end{equation*}
for a.e. $t>0$, and hence
\begin{equation*}
t-\tau_i(t)\ge -\tau_i(0)\qtq{for all}t\ge 0.
\end{equation*}
Therefore, setting
\begin{equation}\label{15qqq}
\tau^*:=\max_{i=1,..., l} \set{\tau_i(0)}
\end{equation}
we may consider the following abstract model:

\begin{equation}\label{16qqq}
\begin{cases}
U'(t)=AU(t)+\sum_{i=1}^lk_i(t)B_iU(t-\tau_i(t))\qtq{in}(0,\infty),\\
U(0)=U_0,\\
B_iU(t)=f_i(t)\qtq{for}t\in [-\tau^*, 0],
\end{cases}
\end{equation}
where the operator $A$, as before, generates an exponentially stable semigroup $(S(t))_{t\ge 0}$ in a Hilbert space $H$, and for each $i=1,\dots, l$,  $B_i$ is a  continuous linear operator of $H$ into itself, $k_i:[0,\infty)\to\RR$ belongs to ${\mathcal L}^1_{loc}([0,+\infty);\RR )$, and $\tau_i$ is a  variable time delay.
Under some mild assumptions on the involved functions and parameters we will establish the  well-posedness of the problem \eqref{16qqq}, and we will obtain exponential decay estimates for its solutions.

The nonlinear version of previous model

\begin{equation}\label{51qqq}
\begin{cases}
U'(t)=AU(t)+\sum_{i=1}^lk_i(t)B_iU(t-\tau_i(t))+F(U(t))\\
\hspace{7,6 cm}\qtq{in}(0,\infty),\\
U(0)=U_0,\\
B_iU(t)=f_i(t)\qtq{for}t\in [-\tau^*, 0],
\end{cases}
\end{equation}
is also analyzed under some Lipschitz continuity assumption on the nonlinear function $F.$ Exponential  decay of the energy  is obtained
for small initial data under a suitable well-posedness assumption. A quite general class of examples satisfying our abstract setting is exhibited.

The paper is organized as follows.
In Sections \ref{s2qqq} and \ref{s3qqq} we prove the well-posedness and the exponential decay estimate for the model \eqref{11} with a single constant time delay.
Next, in Section \ref{s4qqq} we  analyze the more general system \eqref{16qqq} with multiple time-varying time delays. The nonlinear abstract model \eqref{51qqq} is studied in Section \ref{s5qqq}, where we give an exponential decay estimate for small initial data.  Finally, in Section \ref{s6qqq} we  give some applications of our abstract results to concrete models.

\section{Single constant delay: well-posedness}\label{s2qqq}

The following well-posedness result holds true.

\begin{proposition}\label{p21qqq}
Given $U_0\in H$ and a continuous function $f:[0,\tau]\to H$, the problem \eqref{11} has a unique (weak) solution given by Duhamel's formula
\begin{equation}\label{21qqq}
U(t)=S(t)U_0+\int_0^tk(s)S(t-s)BU(s-\tau)\ ds\qtq{for all}t\ge 0.
\end{equation}
\end{proposition}

\begin{proof}
We  proceed step by step by working on time intervals of length $\tau.$
First we consider $t\in [0,\tau]$.
Setting $G_1(t) =k(t)BU(t-\tau),$ $t\in [0,\tau]$ we observe that $G_1(t)= k(t)f(t),$ $t\in [0,\tau]$.
Then, problem \eqref{11} can be rewritten, in the time interval $[0,\tau]$, as a standard inhomogeneous evolution problem:
\begin{equation}\label{22qqq}
\begin{cases}
U'(t)=AU(t)+G_1(t) \qtq{in}(0,\tau),\\
U(0)=U_0.
\end{cases}
\end{equation}

Since $k\in L^1_{loc}([0,\infty);\RR)$ and $f\in C([0,\tau];H)$, we have that $G_1\in L^1((0,\tau);H)$.
Therefore, applying  \cite[Corollary 2.2]{Pazy} there exists a unique solution $U \in C([0,\tau]; H)$ of \eqref{22qqq} satisfying Duhamel's formula
\begin{equation*}
U(t)= S(t)U_0+\int_0^tS(t-s) G_1 (s) ds,\quad t\in [0,\tau]
\end{equation*}
and hence
\begin{equation*}
U(t)= S(t)U_0+\int_0^tk(s) S(t-s) BU (s-\tau) ds,\quad t\in [0,\tau].
\end{equation*}
Next we consider the time interval $[\tau, 2\tau]$.
Setting $G_2(t)=$$k(t)BU(t-\tau)$ we observe that $U(t-\tau)$ is known for $t\in [\tau, 2\tau]$  from the first step.
Then $G_2$ is a known function and it belongs to $L^1((\tau,2\tau); H)$.
So we can rewrite our model \eqref{11} in the time interval $[\tau,2\tau ]$ as the inhomogeneous evolution problem
\begin{equation}\label{23qqq}
\begin{cases}
U'(t)=AU(t)+G_2(t) \qtq{in}(\tau ,2\tau),\\
U(\tau)=U (\tau^-).
\end{cases}
\end{equation}

By the standard theory of abstract Cauchy problems we have a unique continuous solution $U:[\tau, 2\tau)\rightarrow H$ satisfying
\begin{equation*}
U(t)= S(t-\tau )U(\tau^-) +\int_\tau ^tS(t -s) G_2 (s) ds,\quad t\in [\tau ,2\tau],
\end{equation*}
and hence
\begin{equation*}
U(t)= S(t-\tau )U(\tau^-) +\int_\tau ^tk(s) S(t -s) BU (s-\tau) ds,\quad t\in [\tau ,2\tau].
\end{equation*}
Putting together the partial solutions obtained in the first and second steps we get a unique continuous solution $U:[0,2\tau]\rightarrow \RR$
satisfying  Duhamel's formula
\begin{equation*}
U(t)= S(t)U_0+\int_0^tk(s) S(t-s) BU (s-\tau) ds,\quad t\in [0,2\tau].
\end{equation*}

Iterating this argument we find a unique solution $U\in C([0,\infty);H)$ satisfying the representation formula \eqref{21qqq}.
\end{proof}

\section{Single constant delay: exponential stability}\label{s3qqq}

It follows from Duhamel's formula \eqref{21qqq}
 that
\begin{equation}\label{31qqq}
\begin{array}{l}
\displaystyle{e^{\omega t}\norm{U(t)}
\le M\norm{U_0}+Me^{\omega\tau}\int_0^{\tau}\abs{k(s)}e^{\omega (s-\tau)}\norm{f(s)}\ ds}\\
\qquad\displaystyle{
+M\norm{B}e^{\omega\tau}\int_{\tau}^t\abs{k(s)}e^{\omega (s-\tau)}\norm{U(s-\tau)}\ ds}
\end{array}
\end{equation}
for all $t\ge 0$.
Setting
\begin{align*}
&u(t):=e^{\omega t}\norm{U(t)},\\
&\alpha:=M\norm{U_0}+Me^{\omega\tau}\int_0^{\tau}\abs{k(s)}e^{\omega (s-\tau)}\norm{f(s)}\ ds\intertext{and}
&\beta(t):=M\norm{B}e^{\omega\tau}\abs{k(t+\tau)}
\end{align*}
we may rewrite \eqref{31qqq} in the form
\begin{equation*}
u(t)\le \alpha+\int_0^{t-\tau}\beta(s)u(s)\ ds\qtq{for all}t\ge 0.
\end{equation*}

Since $\beta\ge 0$ and $u\ge 0$, it follows that
\begin{equation*}
u(t)\le \alpha+\int_0^{t}\beta(s)u(s)\ ds\qtq{for all}t\ge 0.
\end{equation*}
Applying Gronwall's inequality we conclude that
\begin{equation*}
u(t)\le \alpha e^{\int_0^t\beta(s)\ ds}\qtq{for all}t\ge 0,
\end{equation*}
i.e.,
\begin{equation}\label{32qqq}
\norm{U(t)}\le \alpha e^{\int_0^t\beta(s)\ ds-\omega t}\qtq{for all}t\ge 0.
\end{equation}

This estimate yields the following result:

\begin{theorem}\label{t31qqq}
Assume that there exist two constants  $\omega'\in(0,\omega)$ and $\gamma\in\RR$ such that
\begin{equation}\label{33qqq}
M\norm{B}e^{\omega\tau}\int_0^t \abs{k(s+\tau)}\ ds\le\gamma+\omega't\qtq{for all}t\ge 0.
\end{equation}
Then there exists a constant $M'>0$ such that the solutions of \eqref{11} satisfy the estimate
\begin{equation}\label{34qqq}
\norm{U(t)}\le M'e^{-(\omega-\omega') t}\qtq{for all}t\ge 0.
\end{equation}
\end{theorem}

\begin{proof}
Using our previous notation $\beta(s)=M\norm{B}e^{\omega\tau}\abs{k(s+\tau)}$ we have
\begin{equation*}
\int_0^t\beta(s)\ ds-\omega t\le\gamma-(\omega-\omega') t\qtq{for all}t\ge 0.
\end{equation*}
Combining this with \eqref{32qqq} the estimate \eqref{34qqq} follows with $M':=\alpha e^{\gamma}$.
\end{proof}

\begin{remark}
The hypothesis \eqref{33qqq} is satisfied, in particular,
if the feedback coefficient $k$ belongs to $L^\infty (0,\infty)$
and
\begin{equation*}
M\norm{B}\cdot\norm{k}_\infty e^{\omega\tau} < \omega.
\end{equation*}
It is also satisfied if $k\in L^1(0,\infty)$ or, more generally, if $k=k_1+k_2$ with $k_1\in L^1(0,\infty)$ and $k_2\in L^\infty (0,\infty)$ with a sufficiently small norm $\norm{k_2}_\infty$.
Thus Proposition \ref{t31qqq} extends the results obtained in
\cite{NicaisePignotti18}, in the linear setting, for constant $k$.
\end{remark}

\section{Time variable delays}\label{s4qqq}

Let us now consider the model \eqref{16qqq} with multiple time varying delays.
First, we study its well-posedness.

\begin{theorem}\label{t41}
Given $U_0\in H$ and  continuous functions $f_i:[-\tau^*, 0]\to H$, $i=1,\dots,l,$  the problem \eqref{16qqq} has a unique (weak) solution given by Duhamel's formula
\begin{equation}\label{41qqq}
U(t)=S(t)U_0+\int_0^tS(t-s)\sum_{i=1}^lk_i(s)B_iU(s-\tau_i(s))\ ds,
\end{equation}
for all $t\ge 0.$
\end{theorem}

\begin{proof}
First we consider the case  with time delays functions $\tau_i(\cdot)$  bounded  from below by some positive constants, i.e., there exists for each $i=1,\dots,l$ a constant $\underline{\tau}_i>0$  such that
\begin{equation*}
\tau_i(t)\ge\underline{\tau}_i\qtq{for all}t\ge 0.
\end{equation*}
Then we may argue step by step, as in the proof of Proposition \ref{p21qqq}, by restricting ourselves each time to a time interval of length
\begin{equation*}
\tau_{min}:=\min \set{\underline {\tau}_i\ :\ i=1, \dots, l}>0.
\end{equation*}
Indeed, we infer from the assumption \eqref{13qqq} that
\begin{equation*}
t-\tau_i(t) \le t-\underline\tau_i\le t-\tau_{min}
\end{equation*}
for all $t\ge 0$ and $i=1,\dots,l$.
Therefore, if $t\in [k \tau_{min}, (k+1)  \tau_{min}],$ then $t-\tau_i(t) \le k\tau_{min}$ for all $i=1,\dots, l$.

Now, we pass to the general case.
For each fixed positive number $\epsilon\le 1$ we set
\begin{equation}\label{42qqq}
\tau_i^\epsilon (t):=\tau_i(t)+\epsilon,\quad t\ge0, \quad i=1,\dots, l.
\end{equation}
Moreover, we extend the initial data $f_i$ to continuous functions $\tilde f_i: [-\tau^* -1, 0]\rightarrow H$ with the constant $\tau^*$ defined in \eqref{15qqq}.

Since $\tau_i^\epsilon(t)\ge \epsilon >0$ for all $t$ and $i$, the corresponding model
\eqref{16qqq} with initial data $U_0,$ $\tilde f_i$ and time delay functions $\tau^{\epsilon}_i(\cdot)$  has a unique solution $U_\epsilon (\cdot)\in C[0, +\infty)\rightarrow H$ satisying the representation formula \eqref{41qqq}.

Now consider two positive parameters $\epsilon_1,\epsilon_2\le 1$.
Since $U_{\epsilon_1}(\cdot)$ and $U_{\epsilon_2}(\cdot)$ satisfy \eqref{41qqq}, we have
\begin{multline*}
U_{\epsilon_1}(t)- U_{\epsilon_2}(t) =\\
\int_0^t S(t-s) \sum_{i=1}^l k_i(s)\left [ B_iU_{\epsilon_1}(s-\tau_i^{\epsilon_1}(s))- B_iU_{\epsilon_2}(s-\tau_i^{\epsilon_2}(s))
 \right ]\, ds
\end{multline*}
and hence
\begin{multline*}
U_{\epsilon_1}(t)- U_{\epsilon_2}(t) =\\
\int_0^t S(t-s) \sum_{i=1}^l k_i(s)\left [ B_iU_{\epsilon_1}(s-\tau_i^{\epsilon_1}(s))- B_iU_{\epsilon_2}(s-\tau_i^{\epsilon_1}(s))
 \right ]\, ds\\
 + \int_0^t S(t-s) \sum_{i=1}^l k_i(s)\left [ B_iU_{\epsilon_2}(s-\tau_i^{\epsilon_1}(s))- B_iU_{\epsilon_2}(s-\tau_i^{\epsilon_2}(s))
 \right ]\, ds.
\end{multline*}
It follows that
\begin{equation}\label{43qqq}
e^{\omega t}\Vert U_{\epsilon_1}(t)- U_{\epsilon_2}(t)\Vert  \le M(I_1+I_2)
\end{equation}
with
\begin{equation*}
I_1:=\sum_{i=1}^l\int_{\tau_i^{\epsilon_1}(0)}^t e^{-\omega s}  \vert k_i\vert \vert B_i\Vert\,\Vert U_{\epsilon_1}(s-\tau_i^{\epsilon_1}(s))- U_{\epsilon_2}(s-\tau_i^{\epsilon_1}(s))
 \Vert\, ds
\end{equation*}
and
\begin{equation*}
I_2:=\sum_{i=1}^l \int_0^t e^{-\omega s}  \vert k_i\vert \Vert B_i\Vert\,\Vert  B_iU_{\epsilon_2}(s-\tau_i^{\epsilon_1}(s))- B_iU_{\epsilon_2}(s-\tau_i^{\epsilon_2}(s))
 \Vert \, ds.
\end{equation*}

Using the changes of variable
\begin{equation*}
s-\tau_i^{\epsilon_1}(s)=\sigma
\end{equation*}
in the integrals in the sum $I_1$ and using the notation
\begin{equation}\label{44qqq}
\varphi_i(s):=s-\tau_i(s)=\epsilon_1+\sigma
\end{equation}
we obtain the estimate
\begin{equation}\label{45qqq}
I_1\le \sum_{i=1}^l \frac {e^{\omega (\overline\tau_i+1)}}{1-c_i}\Vert B_i\Vert \int_0^t e^{\omega s} \vert k_i(\varphi_i^{-1}(s+\epsilon_1)\vert \Vert U_{\epsilon_1}(s)- U_{\epsilon_2}(s)\Vert \,
 ds.
\end{equation}

Furthermore, since $U_{\epsilon_2}(\cdot)\in C([0,+\infty);H)$ is locally uniformly continuous and
\begin{equation*}
\tau_i^{\epsilon_1}(t)- \tau_i^{\epsilon_2}(t)=\epsilon_1-\epsilon_2
\end{equation*}
for all $t$ and $i$, we have for every fixed $T>0$ the estimate
\begin{equation}\label{46qqq}
I_2 \le C(T; \epsilon_1-\epsilon_2),
\end{equation}
with a constant $C(T;\epsilon_1-\epsilon_2)$ tending to zero as $\epsilon_1-\epsilon_2\to 0.$

Using \eqref{45qqq} and \eqref{46qqq} in \eqref{43qqq} and applying Gronwall's Lemma for each fixed $T>0,$ we  conclude that
for $\epsilon\rightarrow 0$ the functions $U_\epsilon(\cdot)$ converge locally uniformly to a function $U\in C([0,+\infty); H)$ which satisfies \eqref{41qqq}. This completes the proof.
\end{proof}

Under an appropriate relation between the problem's parameters the system \eqref{16qqq} is exponentially stable:

\begin{theorem}\label{t42qqq}
Assume that there exist two constants $\omega'\in(0,\omega)$ and $\gamma\in\RR$ such that
\begin{equation}\label{47qqq}
M \sum_{i=1}^l\frac {e^{\omega\overline \tau_i}}{1-c_i}\Vert B_i\Vert \int_0^t\abs{k_i(\varphi_i^{-1}(s))} \ ds\le \gamma+\omega^\prime t \quad\mbox{for all}\quad t\ge 0.
\end{equation}
Then there exists a constant $M'>0$ such that the solutions of \eqref{16qqq} satisfy the estimate
\begin{equation}\label{48qqq}
\norm{U(t)}\le M'e^{-(\omega-\omega') t}\qtq{for all}t\ge 0.
\end{equation}
\end{theorem}

\begin{proof} It follows from  Duhamel's formula \eqref{41qqq} that
\begin{align*}
e^{\omega t}\norm{U(t)}
&\le M\norm{U_0}+M\sum_{i=1}^l \int_0^{t}e^{\omega s}\vert k_i(s)\vert  \,\Vert B_i U(s-\tau_i)\Vert\ ds\\
&\le M\norm{U_0}
+M\sum_{i=1}^l e^{\omega\overline\tau_i}\int_{0}^t e^{\omega (s-\tau_i(s))}\vert k_i(s)\vert \,\Vert B_i U(s-\tau_i)\Vert\ ds
\end{align*}
for all $t\ge 0$.

Now we make the change of variable  $\varphi_i(s):= s-\tau_i(s)=\sigma$ for every $i=1,\dots, l$.
Note that the functions $\varphi_i(\cdot) $ are invertible by \eqref{14qqq}.
We have the estimates
\begin{multline*}
\int_0^t e^{\omega (s-\tau_i(s))}\vert k_i(s)\vert  \,\Vert B_i U(s-\tau_i)\Vert\ ds\\
\le \frac 1 {1-c_i}\int_{-\tau_i(0)}^{t-\tau_i(t)} e^{\omega\sigma}\vert k_i(\varphi_i^{-1}(\sigma)\vert\,\Vert B_i U(\sigma )\Vert\ d\sigma
\end{multline*}
and hence
\begin{multline*}
e^{\omega t}\norm{U(t)}
\le M\norm{U_0}+M\sum_{i=1}^l \frac {e^{\omega\overline\tau_i}} {1-c_i} \int_{-\tau_i(0)}^{0}e^{\omega s}\vert k_i(\varphi_i^{-1}(s))\vert  \,\Vert f_i(s)\Vert\ ds\\
+M\sum_{i=1}^l \frac {e^{\omega\overline\tau_i}}{1-c_i}\int_{0}^t e^{\omega s}\vert k_i(\varphi_i^{-1}(s))\vert \,\Vert B_i \Vert \norm{U(s)}\ ds
\end{multline*}
for all $t\ge 0$.
Setting
$\tilde u(t):=e^{\omega t}\norm{U(t)}$ and
\begin{equation}\label{49qqq}
\tilde \alpha:=M\Big  (\norm{U_0}+
\sum_{i=1}^l \frac {e^{\omega\overline\tau_i}} {1-c_i} \int_{-\tau_i(0)}^{0}e^{\omega s}\vert k_i(\varphi_i^{-1}(s))\vert  \,\Vert f_i(s)\Vert\ ds
\Big ),
\end{equation}
\begin{equation}\label{410qqq}
\tilde \beta (t)=M \sum_{i=1}^l \frac {e^{\omega\overline\tau_i}}{1-c_i}\vert k_i(\varphi_i^{-1}(s))\vert \Vert B_i\Vert,
\end{equation}
we may rewrite the above estimate in the form
\begin{equation*}
\tilde u(t)\le \tilde \alpha+\int_0^{t}\tilde \beta(s)\tilde u(s)\ ds\qtq{for all}t\ge 0.
\end{equation*}
Applying Gronwall's inequality we conclude that
\begin{equation*}
\tilde u(t)\le \tilde \alpha e^{\int_0^t\tilde \beta(s)\ ds}\qtq{for all}t\ge 0,
\end{equation*}
i.e.,
\begin{equation}\label{411qqq}
\norm{U(t)}\le \alpha e^{\int_0^t\tilde \beta(s)\ ds-\omega t}\qtq{for all}t\ge 0.
\end{equation}
Now we can conclude as in the proof of Proposition \ref{t31qqq} provided that \eqref{47qqq} is satisfied.
\end{proof}

\begin{remark}
The hypothesis \eqref{47qqq} is satisfied in particular if the feedback coefficients $k_i, i=1,\dots, l,$ belong to $L^\infty (0,+\infty )$ and
\begin{equation*}
M \sum_{i=1}^l\frac {e^{\omega\overline \tau_i}}{1-c_i}\Vert B_i\Vert\,\Vert k_i\Vert_\infty <\omega.
\end{equation*}
It is also satisfied if $k_i(\varphi^{-1}_i(\cdot ))\in L^1(0,+\infty),$ $i=1,\ldots, l,$ or, more generally if $k_i\circ\varphi^{-1}_i =k_i^1+k_i^2$ with $k_i^1\in L^1(0,+\infty)$ and $ k_i^2\in L^\infty (0,+\infty )$ with sufficiently small norms $\Vert k_i^2\Vert_\infty$.
\end{remark}

\section{A nonlinear model}\label{s5qqq}

We now consider the nonlinear model
\eqref{51qqq},
where the operator $A$, as before, generates an exponentially stable semigroup $(S(t))_{t\ge 0}$ in a Hilbert space $H$, and for each  $i=1,\dots, l,$
\begin{itemize}
\item $B_i$ is a  continuous linear operator of $H$ into itself,
\item $k_i\in {\mathcal L}^1_{loc}([0,+\infty);\RR )$,
\item $\tau_i$ is a  variable time delay satisfying \eqref{13qqq} and \eqref{14qqq},
\item $f_i\in C([-\tau^*,0]; H)$ with  $\tau^*$ defined in  \eqref{15qqq}.
\end{itemize}
Furthermore, the nonlinear function $F$ satisfies some local Lipschitz assumption as precised below.

We will prove an exponential stability result for {\em small} initial data under a suitable well-posedness assumption. Then, we will give a class of examples for which this assumption is satisfied.

Concerning $F$ we assume that $F(0)=0$, and that the following local Lipschitz condition is satisfied: for each constant $r>0$ there exists a constant $L(c)>0$  such that
\begin{equation}\label{52qqq}
\norm{F(U)-F(V)}_H\le L(r)\norm{U-V}_H
\end{equation}
whenever $\norm{U}_H\le r$ and  $\norm{V}_H\le r$.

Furthermore, we assume tha the system is well posed in the following sense:

\begin{wpa}
There exist two constants $\omega'\in(0,\omega)$ and $\gamma\in\RR$ such that
\eqref{47qqq} is satisfied.
Furthermore, there exist two positive constants $\rho$ and $C_{\rho}$ with $L(C_\rho ) <\frac {\omega -\omega'}M$ such that if $U_0\in H$ and $f_i\in C([0,\tau^*];H)$ for $i=1, \ldots,l$ satisfy the inequality
\begin{equation}\label{53qqq}
\norm{U_0}^2_H+ \sum_{i=1}^l \int_0^{\tau^*}\abs{k_i(s)} \cdot\norm{f_i(s)}_H^2\ ds < \rho^2,
\end{equation}
then \eqref{51qqq} has a unique solution $U\in C([0,+\infty );H)$ satisfying $\norm{U(t)}\le C_\rho$ for all $t>0$.
\end{wpa}

\begin{theorem}\label{t51qqq}
Under the above well-posedness assumption,
there exists a constant $\tilde M>0$ such that all these solutions satisfy the estimate
\begin{equation*}
\norm{U(t)}\le \tilde M e^{-(\omega-\omega'-ML(C_\rho)) t}\qtq{for all}t\ge 0.
\end{equation*}
\end{theorem}

\begin{proof} By  Duhamel's formula \eqref{41qqq} we have
\begin{multline*}
e^{\omega t}\norm{U(t)}
\le M\norm{U_0}
+M\sum_{i=1}^l e^{\omega\overline\tau_i}\int_{0}^t e^{\omega (s-\tau_i(s))}\abs{k_i(s)} \cdot\norm{B_i U(s-\tau_i)}\ ds\\
+M L(C_\rho) \int_0^t e^{\omega s} \norm{U(s)} \, ds
\end{multline*}
for all $t\ge 0$.

Now, as before, we make the change of variable  $\varphi_i(s):= s-\tau_i(s)=\sigma$ for every $i=1,\dots, l$ to get the estimate
\begin{multline*}
e^{\omega t}\norm{U(t)}
\le M\norm{U_0}+M\sum_{i=1}^l \frac {e^{\omega\overline\tau_i}} {1-c_i} \int_{-\tau_i(0)}^{0}e^{\omega s}\vert k_i(\varphi_i^{-1}(s))\vert\cdot\Vert f_i(s)\Vert\ ds\\
+M\sum_{i=1}^l \frac {e^{\omega\overline\tau_i}}{1-c_i}\int_{0}^t e^{\omega s}\vert k_i(\varphi_i^{-1}(s))\vert\cdot\Vert B_i \Vert\cdot\norm{U(s)}\ ds\\
 +M L(C_\rho) \int_0^t e^{\omega s} \norm{U(s)} \, ds\hspace{3 cm}
\end{multline*}
for all $t\ge 0$.
Setting $\tilde u(t)=e^{\omega t}\Vert U(t)\Vert$
we may rewrite it in the form
\begin{equation*}
\tilde u(t)\le \tilde \alpha+\int_0^{t}\tilde \beta(s)\tilde u(s)\ ds+M L(C_\rho) \,\int_0^t\tilde u(s)\, ds \qtq{for all}t\ge 0,
\end{equation*}
where $\tilde \alpha, \tilde \beta$ are defined as in \eqref{49qqq} and \eqref{410qqq}.
Applying Gronwall's inequality we conclude that
\begin{equation*}
\tilde u(t)\le \tilde \alpha e^{\int_0^t \left [\tilde \beta(s)+ ML(C_\rho)\right ]\ ds}\qtq{for all}t\ge 0,
\end{equation*}
i.e.,
\begin{equation}\label{54}
\norm{U(t)}\le \tilde\alpha e^{\int_0^t\tilde \beta(s)\ ds +M L(C_\rho)t -\omega t}\qtq{for all}t\ge 0.
\end{equation}
Now we can conclude as in the proof of Theorem \ref{t31qqq} provided that \eqref{47qqq} is satisfied.
\end{proof}

Now we consider a class of examples for which the above well-posedness assumption is satisfied.

Let $W$ be a real Hilbert space and  $A_0: {\mathcal D}(A_0)\rightarrow W$  a positive self-adjoint operator with a compact inverse in $W$.
We denote by $\tilde W:=  {\mathcal D}(A_0^{1/2})$ the domain of the operator $A_0^{1/2}.$
Furthermore, let $W_0, \ldots, W_l$  be  real Hilbert spaces
and  $C:W_0\rightarrow W,$ $D_i:W\rightarrow W_i,$ bounded linear operators such that
\begin{equation}\label{55}
\Vert D_i^*u\Vert_{W_i}^2\le d_i\Vert u\Vert_W^2
\end{equation}
and
\begin{equation}\label{56}
a\Vert u\Vert^2_{W_i}\le \Vert C^*u\Vert^2_{W_0}
\end{equation}
for all $u\in W$ and $i=1,\dots, l$, with
suitable positive constants $d_i$ and $a$.

Furthermore, let $ \Psi : \tilde W\rightarrow \RR$ be a functional having a G\^{a}teaux derivative $D \Psi (u)$ in every $u\in \hat W.$
In the same spirit as in \cite{ACS}, we assume the following:

\begin{enumerate}[\upshape (i)]
\item For every $u\in\tilde W$ there is a constant $c(u)$ such that
\begin{equation*}
\vert D \Psi (u) (v)\vert\le c(u)\Vert v\Vert_W,\quad \forall\ v\in\tilde W,
\end{equation*}
where $D\Psi (u)$ is the G\^ateaux derivative of the functional $\Psi$ at $u.$ Thus, $\Psi$ can be extended to the whole $W$ and we denote by $\nabla\Psi (u)$ the unique vector representing $D\Psi (u)$ in the Riesz isomorphism, namely
\begin{equation*}
\langle \nabla\Psi (u), v\rangle_W=D\Psi (u) (v),\quad \forall\ v\in W.
\end{equation*}
\item For all $r>0$ there exists a constant $L(r)>0$ such that
\begin{equation*}
\Vert \nabla\Psi (u)-\nabla\Psi (v)\Vert_W\le L(r)\Vert A_0^{1/2}(u-v)\Vert_W,
\end{equation*}
for all $u,v\in \hat W$ satisfying $\Vert  A_0^{1/2}u\Vert_W\le r$ and $\Vert  A_0^{1/2}v\Vert_W\le r.$
\item $\Psi (0)=0, \nabla\Psi (0)=0,$ and there exists an increasing continuous function $ h$ such that, $\forall \ u \in \tilde W,$
\begin{equation*}
\Vert \nabla \Psi (u)\Vert_W \le h(\Vert A_0^{1/2}u\Vert_W)\Vert A_0^{1/2}u\Vert_W
\end{equation*}
for all  $\ u \in \tilde W$.
\end{enumerate}

In this framework, let us consider the following second-order model:

\begin{equation}\label{57}
\begin{cases}
u_{tt}+ A_0 u + C C^* u_t=\nabla \Psi (u) +\sum_{i=1}^l k_iD_iD_i^*u_t (t-\tau_i(t)), \\
\hspace{7,6 cm}\qtq{in}(0,\infty),\\
u(0)=u_0,\ u_t(0)=u_1, \\
D_i^*u_t(t)=g_i(t)\qtq{for}t\in [-\tau^*, 0],\ i=1,\dots, l,
\end{cases}
\end{equation}
with $(u_0,u_1)\in \tilde W\times W,$ $g_i \in C([-\tau^*,0]; W).$ Here,  the functions $k_i(\cdot),$ $\tau_i(\cdot),$  and the constant $\tau^*$ as defined as before.
In order to
 deal with the nonlinear model we assume
\begin{equation*}
\tau_i(t)\ge\underline{\tau}_i,\quad i=1,\dots,l,
\end{equation*}
and let us set
\begin{equation}\label{58}
\tau_{min}=\min\  \{\, \underline{\tau}_i\,:\, i=1,\dots,l\,\}.
\end{equation}

If we denote $v:=u_t$ and $U:= (u,v)^T,$ \eqref{57} can be recast in the more abstract form \eqref{51qqq} where operator $A$ is defined in $H:= \tilde W\times W$ by
\begin{equation*}
A:=
\begin{pmatrix}
0&1\\
-A_0&-CC^*
\end{pmatrix}
,
\end{equation*}
while $F$ and $B_i$ for $i=1,\dots, l$ are defined by
\begin{equation*}
F(U):=(0,\nabla \Psi (u))^T
\qtq{and}
B_i U:=(0, D_iD_i^*v)^T.
\end{equation*}
Under some conditions on the damping operator $CC^*$ (see, e.g., \cite{BLR} or \cite[Chapter 5]{Komornikbook}) we know that $A$ generates an exponentially stable  semigroup.
Moreover, the above assumptions on $\Psi$  imply that $F(0)=0$ and $F$ satisfies \eqref{52qqq}.

We define the energy functional for the model \eqref{57} as
\begin{multline*}
E(t):= E(t, u(\cdot ))= \frac 12 \Vert u_t\Vert_W^2 +\frac 12 \Vert A_0^{1/2}u\Vert_W^2-\Psi (u)\\
+\frac 12\sum_{i=1}^l \frac 1 {1-c_i}\int_{t-\tau_i(t)}^t \vert k_i(\varphi_i^{-1}(s))\vert\cdot \Vert D_i^*u_t(s)\Vert^2_{W_i} ds.
\end{multline*}
We are going to show that the problem \eqref{57} satisfies the well-posedness assumption and the exponential decay estimate of Theorem \ref{t51qqq} for {\em small} initial data under a suitable compatibility condition between the functions $k_i$ and the constant $a$ in \eqref{56}.
We need a preliminary lemma.

\begin{lemma}\label{l52qqq}
Assume that  $k_i(t)=k^1_i(t) + k^2_i(t)$ with
$k^1_i\in L^1([0, +\infty))$ and $k^2_i\in L^\infty (0,+\infty)$ for $i=1,\dots ,l$. Furthermore, assume that
\begin{equation}\label{59}
\Vert k^2_i\Vert_\infty  \le \frac {2a} l\cdot\frac {1-c_i}{2-c_i},\quad i=1,\dots,l.
\end{equation}
Then, for any solution $u$ of problem \eqref{57}, defined on $[0, T)$ for some $T>0$, and satisfying $E(t)\ge \frac 1 4 \Vert u_t(t)\Vert^2$ for all $t\in [0, T),$  we have
\begin{equation}\label{510}
E(t)\le\bar{C} E(0)\qtq{for all}t\in [0,T)
\end{equation}
with
\begin{equation}\label{511}
\bar{C}= e^{2\sum_{i=1}^l d_i\int_0^{+\infty}
( \frac 1 {1-c_i} \vert k_i^1(\varphi_i^{-1}(s))\vert + \vert k_i^1(s)\vert )\ ds
}.
\end{equation}
\end{lemma}

\begin{proof}
Differentiating, we have
\begin{align*}
E^\prime(t)
&= -\Vert C^*u_t(t)\Vert_{W_0}^2 +\sum_{i=1}^lk_i(t)\langle D_i^*u_t(t), D_i^* u_t(t-\tau_i(t))\rangle \\
&\qquad +\sum_{i=1}^l \frac {\vert k_i(\varphi^{-1}(t)\vert }{2 (1-c_i)}\Vert D_i^*u_t(t)\Vert^2_{W_i}\\
&\qquad -\sum_{i=1}^l \frac {\vert k_i(t)\vert }{2 (1-c_i)}(1-\tau_i^\prime(t))\Vert D_i^*u_t(t-\tau_i(t))\Vert^2_{W_i}.
\end{align*}
Recalling \eqref{14qqq} and using the
 Cauchy--Schwarz inequality hence we infer that
\begin{align*}
E^\prime(t)
&\le  -\Vert C^*u_t(t)\Vert_{W_0}^2 +\frac 1 2 \sum_{i=1}^l \Big ( \frac 1 {1-c_i} \vert k_i(\varphi^{-1}(t))\vert +\vert k_i(t)\vert    \Big )\Vert D_i^*u_t(t)\Vert_{W_i}^2\\
&\le -\Vert C^*u_t(t)\Vert_{W_0}^2 + \frac 1 2 \sum_{i=1}^l \Big ( \frac 1 {1-c_i} \vert k_i^2(\varphi^{-1}(t))\vert +\vert k_i^2(t)\vert    \Big )\Vert D_i^*u_t(t)\Vert_{W_i}^2\\
&\qquad +\frac 1 2 \sum_{i=1}^l \Big ( \frac 1 {1-c_i} \vert k_i^1(\varphi^{-1}(t))\vert +\vert k_i^1(t)\vert    \Big )\Vert D_i^*u_t(t)\Vert_{W_i}^2,
\end{align*}
and then, using \eqref{56} and \eqref{59}, that
\begin{equation}\label{512}
E^\prime (t) \le \frac 1 2 \sum_{i=1}^l  \Big ( \frac 1 {1-c_i}\vert k_i^1(\varphi_i^{-1}(t)\vert +\vert k_i^1(t)\vert \Big   )\, \Vert D_i^* u_t(t)\Vert_{W_i}^2\,.
\end{equation}
From \eqref{512}, recalling \eqref{55}, we obtain
\begin{equation*}
E(t)\le E(0)+\frac 14 \int_0^t K(s) \Vert u_t(s)\Vert^2\, ds
\end{equation*}
with the notation
\begin{equation*}
K(t):=2\sum_{i=1}^l d_i \Big (\frac 1 {1-c_i} \vert k_i^1(\varphi_i^{-1}(t))\vert +\vert k_i^1(t)\vert \Big )\,.
\end{equation*}
Now Gronwall's inequality yields
\begin{equation*}
E(t)\le E(0) e^{\int_0^t K(s)\ ds },
\end{equation*}
proving \eqref{510} with $\bar{C}$ defined by \eqref{511}.
\end{proof}
\begin{proposition}\label{p53qqq}
Assume that  $k_i(t)=k^1_i(t) + k^2_i(t), i=1,\dots ,l,$ with
$k^1_i\in L^1([0, +\infty))$ and $k^2_i\in L^\infty (0,+\infty).$
Moreover, assume \eqref{59}.
Then, the model \eqref{57} satisfies the well posedness assumption of Theorem \ref{t51qqq}.
\end{proposition}

\begin{proof}
First we restrict ourselves to the time interval $[0, \tau_{min}]$  where $\tau_{min}$ is the constant defined in \eqref{58}.
In such an interval the model can be rewritten in the abstract form

\begin{equation}\label{513}
\begin{cases}
U'(t)=AU(t)+\sum_{i=1}^lk_i(t) G_i(t)+F(U(t))\qtq{in}(0,\infty),\\
U(0)=U_0,
\end{cases}
\end{equation}
where $G_i(t)=(0, g_i(t-\tau_i(t))), $ $i=1,\dots,l.$
Then one can apply the classical theory of nonlinear semigroup to deduce the existence of a unique mild solution defined on a maximal interval $[0, \delta )$ with $\delta\le\tau_{min}.$
We will show that for suitably small initial data the solution is global and it satisfies a certain bound.
Our argument is inspired by \cite{ACS} but here additional difficulties appear due to the fact that, being $k_i(\cdot)$ variable in time, we do not have a decreasing energy.  First we observe that if $h\left ( \Vert A_0^{1/2}u_0\Vert_W \right ) <\frac 12,$ then $E(0)>0.$
Indeed,  we deduce from the assumption (iii) on $\Psi$ that
\begin{multline}\label{514}
\vert \Psi(u) \vert\le \int_0^1\vert \langle \nabla \Psi (su), u\rangle\vert\, ds\\
\le
\Vert A_0^{1/2} u\Vert_W^2\, \int_0^1 h (s\Vert A_0^{1/2} u\Vert_W ) s\, ds
\le \frac 12 h (\Vert A_0^{1/2} u\Vert_W  ) \Vert A_0^{1/2} u\Vert^2_W
\end{multline}
Then we have the estimate
\begin{align*}
E(0)&= \frac 12 \Vert u_1\Vert_W^2 +\frac 12 \Vert A_0^{1/2} u_0\Vert_W^2 -\Psi (u_0)\\
&\qquad +\frac 1 2\sum_{i=1}^l \frac 1 {1-c_i} \int_{-\tau_i(0)}^0 \vert k_i(\varphi_i^{-1}(s))\vert\, \Vert D_i^*u_t(s)\Vert^2_{W_i}\ ds\\
&\ge\frac 12 \Vert u_1\Vert_W^2 +\frac 14 \Vert A_0^{1/2} u_0\Vert_W^2\\
&\qquad \frac 1 2\sum_{i=1}^l \frac 1 {1-c_i} \int_{-\tau_i(0)}^0 \vert k_i(\varphi_i^{-1}(s))\vert\, \Vert D_i^*u_t(s)\Vert^2_{W_i}\ ds>0.
\end{align*}

Now  we prove that if
\begin{equation}\label{515}
h\left ( \Vert A_0^{1/2}u_0\Vert_W \right ) <\frac 12 \quad \mbox{and}\quad h\left ( 2 \bar{C}^{1/2}E^{1/2}(0) \right ) <\frac 12,
\end{equation}
where
$\bar{C}$ is the constant defined in \eqref{511},
then
\begin{multline}\label{516}
E(t)> \frac 14 \Vert u_t(t)\Vert_W^2 +\frac 14 \Vert A_0^{1/2} u(t)\Vert^2_W\\
+\frac 1 4\sum_{i=1}^l \frac 1 {1-c_i} \int_{t-\tau_i(t)}^t \vert k_i(\varphi_i^{-1}(s))\vert\, \Vert D_i^*u_t(s)\Vert^2_{W_i}\ ds
\end{multline}
for all $t\in [0,\delta )$.
Indeed, let $r$ be the supremum of all $s\in [0,\delta )$ such that \eqref{516} holds true for every $t\in [0,s].$
Arguing by contradiction, suppose that$r<\delta.$
Then by continuity we have
\begin{multline}\label{517}
E(r)= \frac 14 \Vert u_t(r)\Vert_W^2 +\frac 14 \Vert A_0^{1/2} u(r)\Vert^2_W\\
+\frac 1 4\sum_{i=1}^l \frac 1 {1-c_i} \int_{r-\tau_i(r)}^r \vert k_i(\varphi_i^{-1}(s))\vert\, \Vert D_i^*u_t(s)\Vert^2_{W_i}\ ds.
\end{multline}
Therefore we infer from \eqref{517} and Lemma \ref{l52qqq} that
\begin{equation*}
h( \Vert A_0^{1/2}u(r)\Vert_W)\le h(2 E^{1/2}(r))\le h( 2 \bar{C}^{1/2} E^{1/2}(0))<\frac 12\,.
\end{equation*}
Using \eqref{514} in the definition of $E(r),$ this gives
\begin{align*}
E(r)&= \frac 12 \Vert u_t(r)\Vert_W^2 +\frac 12 \Vert A_0^{1/2} u(r)\Vert^2_W-\Psi (u(r))\\
&\qquad +\frac 1 2\sum_{i=1}^l \frac 1 {1-c_i} \int_{r-\tau_i(r)}^r \vert k_i(\varphi_i^{-1}(s))\vert\cdot\Vert D_i^*u_t(s)\Vert^2_{W_i}ds
\\
&>\frac 14 \Vert u_t(r)\Vert_W^2 +\frac 14 \Vert A_0^{1/2} u(r)\Vert^2_W\\
&\qquad +
\frac 1 4\sum_{i=1}^l \frac 1 {1-c_i} \int_{r-\tau_i(r)}^r \vert k_i(\varphi_i^{-1}(s))\vert\cdot\Vert D_i^*u_t(s)\Vert^2_{W_i}ds,
\end{align*}
contradicting the maximality of $r.$
This implies $r=\delta.$

Now, let us define
\begin{equation}\label{518}
\rho:=
\frac 1 {2\bar{C}^{1/2}} h^{-1}(\frac 12).
\end{equation}
We show that \eqref{515} is satisfied for all $u_0\in \tilde W,$ $u_1\in W,$ $g_i\in C([-\tau^*,0], W_i),$ $i=1,\dots,l,$ satisfying
\begin{equation*}
\Vert A_0^{1/2}u_0\Vert_W^2+\Vert u_1\Vert_W^2+ \frac 12 \sum_{i=1}^l\frac 1 {1-c_i}\int_{-\tau_i(0)}^0 \vert k_i(\varphi_i^{-1} (s))\vert\, \Vert g_i(s)\Vert^2_{W_i}\, ds
<\rho^2.
\end{equation*}
Indeed, this assumption implies $\Vert A_0^{1/2}u_0\Vert_W <\rho$ and then, observing that $\bar{C}>1,$ we have
$$ h \Big ( \Vert A_0^{1/2}u_0\Vert_W\Big ) < h(\rho )=h \Big (  \frac 1 {2\bar{C}^{1/2}} h^{-1}(\frac 12  )\Big )< \frac 12\,.$$
Moreover, from \eqref{514} we get the estimate
\begin{multline*}
E(0)\le \frac 34 \Vert A_0^{1/2} u_0\Vert_W^2+\frac 12\Vert u_1\Vert_W^2\\
+ \frac 12 \sum_{i=1}^l\frac 1 {1-c_i}\int_{-\tau_i(0)}^0 \vert k_i(\varphi_i^{-1} (s))\vert\, \Vert g_i(s)\Vert^2_{W_i}\, ds<\rho^2,
\end{multline*}
and thus, from \eqref{518} we infer that
\begin{equation*}
h\left ( 2 \bar{C}^{1/2}E^{1/2}(0) \right ) < h (2\bar{C}^{1/2}\rho ) <h(h^{-1}(1/2))= \frac 1 2.
\end{equation*}
We  conclude that \eqref{515} holds, and that
\begin{multline*}
0< \frac 14 \Vert u_t(t)\Vert^2_W +\frac 14 \Vert A_0^{1/2} u(t)\Vert^2_W\\
+\frac 1 4\sum_{i=1}^l \frac 1 {1-c_i} \int_{t-\tau_i(t)}^t \vert k_i(\varphi_i^{-1}(s))\vert\, \Vert D_i^*u_t(s)\Vert^2_{W_i}ds\\
\le E(t)\le \bar{C} E(0)\le \bar{C}\rho^2, \ \forall \ t\in [0,\delta].
\end{multline*}
One can extend the solution of problem (\ref{513}) by considering as initial datum the solution at time $t=\delta\,.$
Arguing as above, we can extend the solution to the whole $[0,\tau_{min}]$ and the solution satisfies
$$h\left (\Vert A_0^{1/2}u(\tau_{min})\Vert_W \right )\le h\left (2 E^{1/2}(\tau_{min})\right )\le h\left (2 \bar{C}^{1/2}E^{1/2}(0) \right )<\frac 12,$$
where we have applied estimate \eqref{510}
on the whole interval $[0,\tau_{min}].$ Once we have the solution $U(\cdot)$ on the interval $[0, \tau_{min}],$ then on the second interval $[\tau_{min}, 2\tau_{min}]$ one  can rewrite  again
our problem in the abstract form  (\ref{513}) with $G_i(t)=(0, D_iD_i^*u_t(t-\tau_i(t)))$ (note that, for $t\in     [\tau_{min}, 2\tau_{min}]$ it results $t-\tau_i(t)\le\tau_{min}$). One can repeat the same argument on every interval of length $\tau_{min}$ to get a global solution satisfying \eqref{510}.
\end{proof}

\section{Examples}\label{s6qqq}

In the following examples we consider a non-empty bounded domain $\Omega$ in $\RR^n$ with a boundary $\Gamma $ of class $C^2$.

\subsection{The wave equation with localized frictional damping}

Let $O\subset \Omega$ be a nonempty open subset satisfying the \emph{geometrical control  property} in \cite{BLR}.
For instance, denoting by $m$ the standard multiplier $m(x)=x-x_0,$ $x_0\in\RR^n,$ as in \cite{Lions-1988}, $O$ can be the intersection of $\Omega$ with an open neighborhood of the set
\begin{equation*}
\Gamma_0=\set{x\in\Gamma \:\ m(x)\cdot\nu (x)>0}.
\end{equation*}
Denoting by $\chi_D$ the characteristic function of a set $D,$
let us consider the following system:

\begin{equation}\label{61qqq}
\begin{cases}
u_{tt}(x,t)-\Delta u(x,t)+a\chi_O(x) u_t(x,t) \\
\hspace{2cm} {-k(t)\chi_{\tilde O}(x) u_t(x,t-\tau)=0}\qtq{in}\Omega\times (0, \infty),\\
{u(x,t)=0\qtq{on}\Gamma\times (0,\infty )},\\
{u(x,t-\tau)=u_0(x, t) \qtq{in}\Omega\times (0,\tau]},\\
{u_t(x,t-\tau)=u_1(x,t)}
{\qtq{in}\Omega\times (0,\tau]},
\end{cases}
\end{equation}
where $a$ is a positive constant,
$\tau >0$ is the time delay, and the damping coefficient $k$ belongs to $L^1_{loc}(0,\infty)$.
The set $\tilde O\subset\Omega$ where the delay feedback is localized can be any measurable subset of $\Omega.$

Setting $U= (u, u_t)^T,$ this problem can be rewritten in the form \eqref{11}
with $H=  H^1_0(\Omega)\times L^2(\Omega),$
\begin{equation*}
A=
\begin{pmatrix}
0&1\\\Delta&\chi_O
\end{pmatrix}
\end{equation*}
and
$B: H\to H$ defined by
\begin{equation*}
B (u,v)^T= (0, \chi_{\tilde O}v)^T.
\end{equation*}
It is well-known that ${A}$ generates a strongly continuous semigroup which is exponentially stable (see e.g. \cite{zuazua, Komornikbook}).
Since $\norm{B}\le 1$, we have exponential stability result under the assumption
\begin{equation}\label{62qqq}
Me^{\omega\tau}\int_0^t\abs{k(t+\tau )}\ ds \le\gamma+\omega't\qtq{for all}t\ge 0
\end{equation}
for some $\omega'<\omega$ and $\gamma\in\RR$, where $M$ and $\omega$ denote the positive constants in the exponential estimate \eqref{12} for the semigroup generated by $A.$
This extends to more general delay feedbacks a previous result of the second author \cite{SCL12}.

\subsection{The wave equation with memory }\label{ss62qqq}
Given an arbitrary open subset $O$ of $\Omega$, we consider the system

\begin{equation}\label{63qqq}
\begin{cases}
u_{tt}(x,t) -\Delta u (x,t)+
\int_0^\infty \mu (s)\Delta u(x,t-s) ds\\
\hspace{2,5 cm}
 =k(t)\chi_O (x) u_t(x,t-\tau) )\qtq{in} \Omega\times
(0,\infty),\\
u (x,t) =0\qtq{on}\Gamma \times
(0,\infty),\\
u(x,t)=u_0(x, t)\qtq{in}\Omega\times (-\infty, 0]
\end{cases}
\end{equation}
with a constant time delay $\tau >0$, and a locally absolutely continuous memory kernel $\mu :[0,\infty)\rightarrow [0,\infty)$,
satisfying the following three conditions:
\begin{enumerate}[\upshape (i)]
\item $\mu (0)=\mu_0>0;$
\item $\int_0^{\infty} \mu (t) dt=\tilde \mu <1;$
\item $\mu^{\prime} (t)\le -\alpha \mu (t) \qtq{for some}\alpha >0.$
\end{enumerate}

As in \cite{Dafermos}, we introduce the notation
\begin{equation*}
\eta^t(x,s):=u(x,t)-u(x,t-s).
\end{equation*}
Then we can restate  \eqref{63qqq}
in the following form:
\begin{equation}\label{64qqq}
\begin{cases}
u_{tt}(x,t)= (1-\tilde \mu)\Delta u (x,t)+
\int_0^\infty \mu (s)\Delta \eta^t(x,s) ds\\
\hspace{2.5 cm}
=k(t)\chi_O (x) u_t(x,t-\tau))\qtq{in} \Omega\times
(0,\infty)\\
\eta_t^t(x,s)=-\eta^t_s(x,s)+u_t(x,t)\qtq{in}\Omega\times
(0,\infty)\times (0,\infty ),\\
u (x,t) =0\qtq{on}\Gamma \times
(0,\infty)\\
\eta^t (x,s) =0\qtq{in}\Gamma \times
(0,\infty) \qtq{for} t\ge 0,\\
u(x,0)=u_0(x)\qtq{and}\quad u_t(x,0)=u_1(x)\qtq{in}\Omega,\\
\eta^0(x,s)=\eta_0(x,s) \qtq{in}\ \Omega\times
(0,\infty),
\end{cases}
\end{equation}
where
\begin{equation}\label{65qqq}
\begin{cases}
u_0(x)=u_0(x,0), \quad x\in\Omega,\\
u_1(x)=\frac {\partial u_0}{\partial t}(x,t)\vert_{t=0},\quad x\in\Omega,\\
\eta_0(x,s)=u_0(x,0)-u_0(x,-s),\quad x\in\Omega,\  s\in (0,\infty).
\end{cases}
\end{equation}

Let us introduce the Hilbert space
$L^2_{\mu}((0, \infty);H^1_0(\Omega ))$
of $H^1_0$-valued functions on $(0,\infty),$
endowed with the inner product
\begin{equation*}
\langle \varphi, \psi\rangle_{L^2_{\mu}((0, \infty);H^1_0(\Omega ))}=
\int_{\Omega}\left (\int_0^\infty \mu (s)\nabla \varphi (x,s)\nabla \psi (x,s) ds\right )dx,
\end{equation*}
and then the the Hilbert space
\begin{equation*}
H:=
H^1_0(\Omega)\times L^2(\Omega)\times L^2_{\mu}((0, \infty);H^1_0(\Omega ))
\end{equation*}
equipped  with the inner product
\begin{multline*}\label{innerd}
\left\langle
\begin{pmatrix}
u\\ v\\ w
\end{pmatrix}
,
\begin{pmatrix}
\tilde u\\ \tilde v\\ \tilde w
\end{pmatrix}
\right\rangle_H
:=  (1-\tilde\mu )\int_\Omega \nabla u\nabla\tilde u dx + \int_\Omega v\tilde v dx\\
+\int_{\Omega} \int_0^\infty \mu (s)\nabla w\nabla\tilde w ds dx.
\end{multline*}

Setting
$U:= (u,u_t,\eta^t)^T$
we may rewrite   the problem \eqref{64qqq}--\eqref{65qqq}
in the abstract form
\begin{equation*}\label{abstractd}
\begin{cases}
U_t(t)={ A} {U}(t)+k B U(t-\tau),\\
U(0)=(u_0,u_1, \eta_0)^T,
\end{cases}
\end{equation*}
where the operator $A$ is defined by
\begin{equation*}\label{Operator}
A
\begin{pmatrix}
u\\ v\\ w
\end{pmatrix}
:=
\begin{pmatrix}
v\\
(1-\tilde\mu)\Delta u+\int_0^{\infty}\mu (s)\Delta w(s)ds\\
-w_s+v
\end{pmatrix}
\end{equation*}
with domain
\begin{align*}\label{dominioOpd}
D(A):=
&\left\{(u,v,\eta )\in   H^1_0(\Omega)\times H^1_0(\Omega)
\times L^2_{\mu}((0,\infty);H^1_0(\Omega))\, :\right.\\
&(1-\tilde\mu)u+\int_0^\infty \mu (s)\eta (s) ds \in H^2(\Omega)\cap H^1_0(\Omega),\\
&\left.\eta_s\in  L^2_{\mu}((0\infty);H^1_0(\Omega))\right\}
\end{align*}
in the Hilbert space $H$,
and the bounded operator $B:H\to H$ is defined by the formula
\begin{equation*}
B(u,v, \eta^t)^T:= (0, \chi_O v, 0)^T.
\end{equation*}

It is well-known (see e.g. \cite{GRP}) that the operator $A$ generates an exponentially stable semigroup.
Since $\norm{B}\le 1$, Theorems \ref{p21qqq} and  \ref{t31qqq} guarantee the well-posedness and exponential stability of \eqref{64qqq}--\eqref{65qqq} if the condition \eqref{62qqq} is satisfied
for some $\omega'<\omega$ and $\gamma\in\RR$, where $M$ and $\omega$ denote the positive constants in the exponential estimate \eqref{12} for the semigroup generated by $A.$

\subsection{The wave equation with frictional damping and source}

Let $O\subset \Omega$ be a nonempty open subset satisfying the \emph{geometrical control  property} in \cite{BLR} and let  $\tilde O\subset O.$
As an explicit example of (\ref{513}), let us consider the following system:

\begin{equation}\label{66qqq}
\begin{cases}
u_{tt}(x,t)-\Delta u(x,t)+a\chi_O(x) u_t(x,t)\\
\hspace{6cm} -k(t)\chi_{\tilde O}(x) u_t(x,t-\tau) \\
\qquad =\vert u(x,t)\vert^\mu u(x,t)\qtq{in}\Omega\times (0, \infty),\\
u(x,t)=0\qtq{on}\Gamma\times (0,\infty ),\\
{u(x,t-\tau)=u_0(x, t) \qtq{in}\Omega\times (0,\tau]},\\
{u_t(x,t-\tau)=u_1(x,t)}
{\qtq{in}\Omega\times (0,\tau]},
\end{cases}
\end{equation}
where $a$ is a positive constant,
$\tau >0$ is the time delay, and the damping coefficient $k$ belongs to $L^1_{loc}(0,\infty)$.
Moreover, we assume $k= k^1+k^2$ with $k_1\in L^1([0,+\infty))$ and $k_2\in L^\infty (0,+\infty)$ satisfying $\Vert k^2\Vert_\infty < a.$
Setting $U= (u, u_t)^T,$ this problem can be rewritten in the form \eqref{51qqq}
with $H=  H^1_0(\Omega)\times L^2(\Omega),$
\begin{equation*}
A=
\begin{pmatrix}
0&1\\\Delta&\chi_O
\end{pmatrix}
\end{equation*}
and
$B: H\to H$ defined by
\begin{equation*}
B (u,v)^T= (0, \chi_{\tilde O}v)^T.
\end{equation*}
It is well-known that ${A}$ generates a strongly continuous semigroup which is exponentially stable (see e.g. \cite{zuazua, Komornikbook}).
Next, consider the functional
$$\Psi (u):= \frac 1 {\mu+2} \int_{\Omega} \vert u(x)\vert^{\mu+2} dx,\quad u\in H_0^1(\Omega),$$
which is well-defined, for $0<\mu\le \frac 4 {n-2},$ by Sobolev's embedding theorem. Note that $\Psi$ is G\^{a}teaux differentiable at every $u\in
H_0^1(\Omega)$ and its G\^{a}teaux  derivative is given by
$$D\Psi (u)(v)= \int_{\Omega} \vert u(x)\vert^\mu u(x) v(x)\, dx,\quad v \in H_0^1(\Omega).$$
As showed in \cite{ACS}, assuming
$0<\mu<\frac 4 {n-2},$ then $\Psi$ satisfies previous assumptions (i), (ii), (iii), and then problem \eqref{66qqq} is included in the abstract form (\ref{513}). Since the assumptions of Lemma \ref{l52qqq} are satisfied then Theorem \ref{t51qqq} holds
for small initial data if
\eqref{62qqq}
is satisfied
for some $\omega'<\omega$ and $\gamma\in\RR$, where $M$ and $\omega$ denote the positive constants in the exponential estimate \eqref{12} for the semigroup generated by $A.$

\begin{remark}
A  large variety of other examples could be considered, e.g., the wave equation with standard dissipative boundary conditions and internal delays, plate equations with internal/boundary/viscoelastic dissipative feedbacks and internal delays,  elasticity systems with different kinds of feedbacks.
\end{remark}

\end{document}